\documentclass[11pt,reqno]{amsart}

\usepackage{amssymb,amsmath,amsthm,amscd,latexsym,amsfonts}
\usepackage{mathtools}
\usepackage[T1]{fontenc}

\usepackage{graphicx}
\usepackage{color}
\usepackage[a4paper,top=3cm, bottom=3cm, left=3.3cm, right=2.3cm]{geometry}

\newtheorem{thm}{Theorem}

\newtheorem{lemma}{Lemma}
\newtheorem{pro}{Proposition}

\numberwithin{equation}{section} 

\newcommand{\bea}{\begin{eqnarray}}
\newcommand{\eea}{\end{eqnarray}}

\begin{document}
\title[Dynamics of a population]{Dynamics of dioecious population with different fitness of genotypes}

\author{A. M. Diyorov, U. A. Rozikov}

\address{A. \ M. \ Diyorov \\ The Samarkand branch of TUIT, Samarkand, Uzbekistan.}
\email {dabduqahhor@mail.ru}

 \address{U.\ A.\ Rozikov\\ V.I.Romanovskiy Institute of mathematics,
81, Mirzo Ulug'bek str., 100170, Tashkent, Uzbekistan.}
\email {rozikovu@yandex.ru}

\begin{abstract} In this paper we study dynamical systems generated by an evolution operator
 of a dioecious population. This evolution operator is a six-parametric,
non-linear operator mapping $[0,1]^2$ to itself. We find all fixed points and
under some conditions on parameters we give limit points of trajectories constructed
by iterations of the evolution operator.
\end{abstract}
\maketitle



\section{Introduction} \label{sec:intro}
In biology a {\it gene} is the molecular unit of heredity of a living organism.
 An {\it allele} is one of a number of alternative forms of the same gene.
A {\it chromosome} contains the genetic code of a gene. {\it Gamete} is a cell that fuses with
another cell during fertilization in organisms that sexually reproduce. A {\it zygote} is a cell
formed by a fertilization event between two gametes.

A {\it dioecy}\footnote{see https://en.wikipedia.org/wiki/Dioecy \ \ and references therein}
is a characteristic of a species, meaning that it has distinct male and female
individual organisms. {\it Dioecious} reproduction is biparental reproduction.
Dioecy is a method that excludes self-fertilization and promotes allogamy (outcrossing),
and thus tends to reduce the expression of recessive deleterious mutations present in a population \cite{C}.

In zoology, dioecious species may be opposed to hermaphroditic species, meaning that an individual is of only
one sex, in which case the synonym gonochory is more often used.
Dioecy may also describe colonies within a species, which may be either dioecious
or monoecious \cite{D}. An individual dioecious colony contains members of only
one sex, whereas monoecious colonies contain members of both sexes.
Most animal species are dioecious.

Following \cite[page 45]{E} we consider a population which admits two sexes.
Study a given gene locus at which two alleles $A_1$ and $A_2$ may occur.
Let $A_1A_1$, $A_1A_2$, and $A_2A_2$ be genotypes of the populations.
Assume that viability selection exists, so that the relative fitness of the genotype
$A_iA_j$ in males is $w_{ij}$, with corresponding values $v_{ij}$ in females.
Consider genotypic frequencies immediately after the formation of the zygotes of any
generation, and suppose that in a given generation the males produce $A_1$ gametes
with frequency $x$ and $A_2$ gametes with frequency $1-x$. Denote the corresponding
frequencies for females by $y$ and $1-y$. At the time of conception of the
zygotes in the daughter generation the genotypic frequencies are, in both sexes,
$$\begin{array}{cccc}
 A_1A_1 & A_1A_2 &  A_2A_2\\[2mm]
 xy & x(1-y)+y(1-x) &  (1-x)(1-y)
\end{array}$$
By the age of maturity these frequencies will have been altered  by
differential viability to the relative values
\begin{equation}\label{bp}
\begin{array}{ccccc}
\ \ \ \ \ \ & A_1A_1 & A_1A_2 & A_2A_2\\[2mm]
{\rm males}:  & w_{11}xy & w_{12}[x(1-y)+y(1-x)] &  w_{22}(1-x)(1-y)\\[2mm]
{\rm females}:  & v_{11}xy & v_{12}[x(1-y)+y(1-x)] &  v_{22}(1-x)(1-y)
\end{array}
\end{equation}

In this paper we consider the population (\ref{bp}).

To define an evolution operator of this population let us give necessary definitions.

The following set is called $(m-1)$-dimensional simplex:
$$ S^{m-1}=\{x=(x_1,...,x_m)\in R^m: x_i\geq 0,
\sum^m_{i=1}x_i=1 \}.$$

A state of population (\ref{bp}) is a pair of probability distributions
$x=(x_1,x_2)\in S^1, \, y=(y_1,y_2)\in S^1$ on the set $\{A_1, A_2\}$.

Denote $S=\{(x,y): 0\leq x\leq 1, 0\leq y\leq 1\}$.

The operator  $W:S\rightarrow S$, corresponding to population (\ref{bp}), is defined by
\begin{equation}\label{w}
\begin{array}{ll}
x'={axy+ b(x(1-y)+y(1-x))\over axy+b(x(1-y)+y(1-x))+c(1-x)(1-y)},\\[3mm]
y'={\alpha xy+ \beta (x(1-y)+y(1-x))\over \alpha xy+\beta (x(1-y)+y(1-x))+\gamma (1-x)(1-y)},
\end{array}
\end{equation}
where $x'$ and $y'$ are the frequencies of $A_1$ gametes produced by males and
females of the daughter generation. Here $a=w_{11}$, $b=w_{12}$, $c=w_{22}$, $\alpha=v_{11}$, $\beta=v_{12}$, and $\gamma=v_{22}$.

This operator is called an evolution operator (\cite{E}, \cite{L}),
where
$$x=x_1, 1-x=x_2, \, y=y_1, 1-y=y_2,$$
$$a,b,c,\alpha,\beta,\gamma\geq 0, a+b\ne 0, \alpha+\beta\ne 0.$$

The operator (\ref{w}) for any initial point (state)  $(x^{(0)},y^{(0)})\in S$ defines its trajectory:
$$
\{(x^{(n)},y^{(n)})\}_{n=0}^\infty :
(x^{(n+1)},y^{(n+1)})=W((x^{(n)},y^{(n)}))=W^{(n+1)}((x^{(0)},y^{(0)})),
n=0,1,2,...$$
i.e.
\begin{equation}\label{wt}
\begin{array}{ll}
x^{(n+1)}={ax^{(n)}y^{(n)}+ b(x^{(n)}(1-y^{(n)})+y^{(n)}(1-x^{(n)}))\over ax^{(n)}y^{(n)}+b(x^{(n)}(1-y^{(n)})+y^{(n)}(1-x^{(n)}))+c(1-x^{(n)})(1-y^{(n)})},\\[3mm]
y^{(n+1)}={\alpha x^{(n)}y^{(n)}+ \beta (x^{(n)}(1-y^{(n)})+
y^{(n)}(1-x^{(n)}))\over \alpha x^{(n)}y^{(n)}+\beta (x^{(n)}(1-y^{(n)})+y^{(n)}(1-x^{(n)}))+\gamma (1-x^{(n)})(1-y^{(n)})}.
\end{array}
\end{equation}

The main problem for a given operator $W$ is to investigate the trajectory $W^{(n)}((x^{(0)},y^{(0)}))$, for any initial point
$(x^0,y^0)$. The difficulty of the problem depends on the given operator $W$. In this paper we study
trajectories given by the operator (\ref{w}).

\section{Analysis of the trajectories}
\subsection{Fixed points.} In this subsection we give all fixed points of $W$. Such points are solutions to
$W((x,y))=(x,y)$, i.e.,
\begin{equation}\label{wfp}
\begin{array}{ll}
x={axy+ b(x(1-y)+y(1-x))\over axy+b(x(1-y)+y(1-x))+c(1-x)(1-y)},\\[3mm]
y={\alpha xy+ \beta (x(1-y)+y(1-x))\over \alpha xy+\beta (x(1-y)+y(1-x))+\gamma (1-x)(1-y)},
\end{array}
\end{equation}
with condition that $x,y\in [0,1]$.

From the first equation of the system (\ref{wfp}) we get
$$(x-1)[\{(a-2b+c)x+b\}y+(b-c)x]=0,$$
hence
$$x=1, \ \ y={(c-b)x\over (a-2b+c)x+b}.$$
Substituting $x=1$ into second equation we get
$y=1$ (for $\alpha\beta>0$ and $\alpha=0$, $\beta\ne 0$) and $y=0$
(for $\alpha\ne 0$, $\beta=0$).
Now substituting  $y={(c-b)x\over (a-2b+c)x+b}$ into second equation
and solving it with respect to $x$, we get the following three solutions:
$$x=0, \ \ x={b\over b-a}, \ \ x={(c-b)\gamma-c\beta\over  (c-b)(\alpha+\gamma)+ (a-c)\beta}.$$
Consequently (\ref{wfp}) has the following solutions
\begin{equation}\label{5}
\begin{array}{lll}
z_0=(0,0), \ \ z_1=\left({b\over b-a},1\right),\\[3mm]
 z_2=\left({(c-b)\gamma-c\beta\over  (c-b)(\alpha+\gamma)+ (a-c)\beta},
 {(\gamma-\beta)c-\gamma b\over  (\gamma-\beta)(a+c)+ (\alpha-\gamma)b}\right),\\[3mm]
  z_3=(1,1), \ \ z_4=(1, {\beta\over \beta-\alpha}).
 \end{array}
 \end{equation}
Note that, for all $a,b>0$,  ${b\over b-a}\notin [0,1]$  i.e.
$z_1\notin S$. Moreover, for $a=0, b\ne 0$ (resp. $b=0, a\ne 0$) we have
 $z_1=(1,1)\in S$ (resp.  $z_1=(0,1)\in S$). Similar conclusions true for $z_4$.

Consider the set of parameters defined by
$$\mathcal P=\{(a,b,c,\alpha,\beta,\gamma)\in \mathbb R^6_+: z_2\in S\}.$$
Since $\{(a,b,c,\alpha,\beta,\gamma)\in \mathbb R^6_+: a=\alpha, b=\beta, c=\gamma, c>2b\}$ is a subset for
$\mathcal P$, we have $\mathcal P\ne\emptyset$.

Summarize the above obtained results  in the following:

\begin{thm} For the operator (\ref{w}) the following hold

\begin{itemize}
\item[1)] If $ab\ne 0$, $\alpha\beta\ne 0$, $(a,b,c,\alpha,\beta,\gamma)\notin \mathcal P$, then  the operator (\ref{w})
has two fixed point $z_0$, $z_3$.

\item[2)]  If $ab=0$, $\alpha\beta\ne 0$,  or  $(a,b,c,\alpha,\beta,\gamma)\in \mathcal P$, then  the operator (\ref{w})
has up to five fixed points $z_i$, $i=0,1,2,3,4.$
\end{itemize}
\end{thm}

\subsection{A symmetric case.}
To simplify the problem of investigation of trajectories for (\ref{w})
we consider the case $a=\alpha>0, b=\beta>0, c=\gamma$. In this case we fully
describe the limit points of each trajectory.

Then the restriction of the operator $W$ on the
invariant set
$$\mathcal M=\{(x,y)\in S: \ x=y\}$$ has the form
\begin{equation}\label{fx}
x^\prime=f(x)\equiv {ax^2+ 2bx(1-x)\over ax^2+2bx(1-x)+c(1-x)^2}.
\end{equation}
This function $f(x)$ has the following fixed points:
$$0, \ \ 1,  \ \ x^*={c-2b\over a-2b+c}.$$
Note that $x^*\in (0,1)$ iff  $c> 2b$.

Denote
$$A(x, y)={axy+ b(x(1-y)+y(1-x))\over axy+b(x(1-y)+y(1-x))+c(1-x)(1-y)}.$$

\begin{thm} If $a=\alpha>0$, $b=\beta>0$, $c=\gamma$ in (\ref{w}) and $(x^0,y^0)\in S$  is an initial point then
\begin{itemize}
\item[(i)] $W(x^0,y^0)\in \mathcal M$,
\item[(ii)] if $c\leq 2b$ then
$$\lim_{n\to \infty}W^n(x^0, y^0)=\left\{\begin{array}{ll}
(0,0), \ \ \mbox{if} \ \ A(x^0, y^0)=0\\[2mm]
(1,1), \ \ \mbox{if} \ \ A(x^0, y^0)>0
\end{array}\right.
$$
\item[(iii)] if $c>2b$ then
$$\lim_{n\to \infty}W^n(x^0, y^0)=\left\{\begin{array}{lll}
(0,0), \ \ \mbox{if} \ \ A(x^0,y^0)<x^*\\[2mm]
(x^*,x^*), \ \ \mbox{if} \ \ A(x^0,y^0)=x^*\\[2mm]
(1,1), \ \ \mbox{if} \ \ A(x^0,y^0)>x^*.
\end{array}\right.
$$
\end{itemize}
\end{thm}
\begin{proof} From conditions of theorem it follows that $x'=y'$, i.e., (i) holds.
Therefore to investigate a trajectory for operator $W$ it suffices  to consider it on the invariant set $\mathcal M$.
Simple analysis of the function (\ref{fx}) shows that it is monotone increasing and
convex when $c\leq 2b$ (see Fig. \ref{f1}), and for $c>2b$
it is concave if $x\in [0,x^*)$ and convex if  $x\in [x^*,1]$ (see Fig. \ref{f2}).
This completes the proof.
\begin{figure}[h]
\begin{center}
\includegraphics[width=8cm]{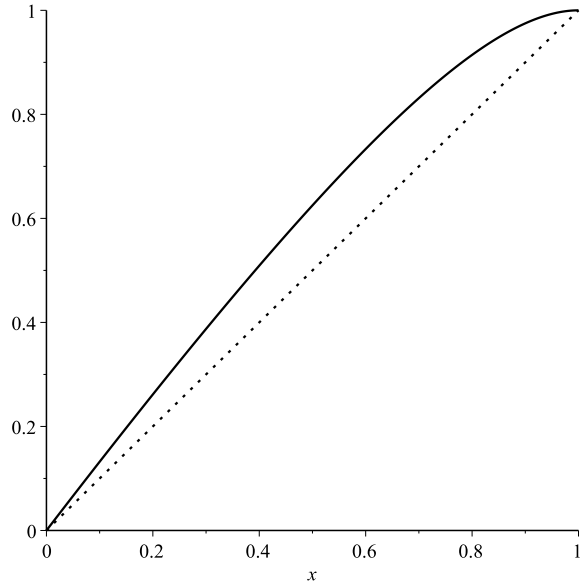}
\end{center}
\caption{The graph of $f$ for $c\leq 2b$.}
\label{f1}
\end{figure}
\begin{figure}[h]
\begin{center}
\includegraphics[width=8cm]{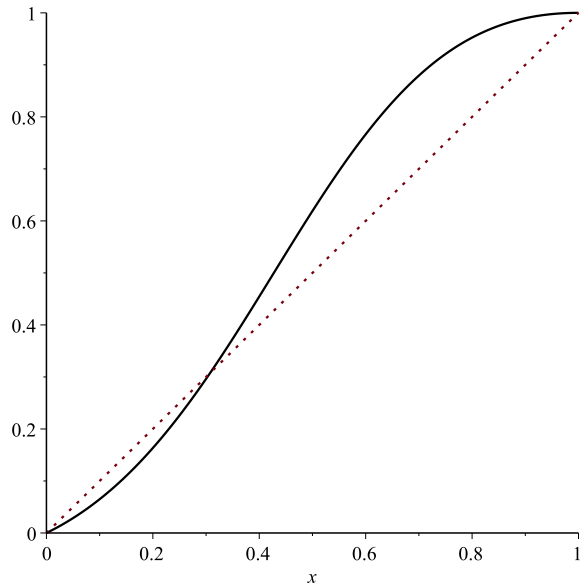}
\end{center}
\caption{The graph of $f$ for $c>2b$.}
\label{f2}
\end{figure}
\end{proof}

%

\subsection{The general case}
If $x=1$ or $y=1$ in (\ref{w}) then $x'=y'=1$.
Therefore for any initial point of the form $(1, y)$ and $(x, 1)$
we have
$$ W(1, y)=W(x, 1)=(1, 1).$$
The case $c\cdot \gamma =0$ also trivially gives $W^2(x, y)=(1, 1).$

Thus we should consider the case  $c\cdot \gamma \ne 0$,
  $x\ne 1$ and $y\ne 1$.

\begin{lemma}\label{q2} If $c\gamma\ne 0$ and an initial point $(x^{(0)}, y^{(0)})\in S$ is
such that $x^{(0)}\ne 1$ or $y^{(0)}\ne 1$
then $x^{(n)}\ne 1$ and $y^{(n)}\ne 1$ for any $n\geq 1$.
\end{lemma}
\begin{proof} Assume that there exists $n_0\geq 0$ such that
\begin{equation}\label{n0}
\begin{array}{ll}
x^{(k)}\ne 1 \ \ {\rm and} \ \ y^{(k)}\ne 1, \ \ {\rm for \, any} \ \ k\leq n_0,\\[2mm]
x^{(n_0+1)}=1 \ \ {\rm or} \ \ y^{(n_0+1)}=1
\end{array}
\end{equation}
 then by (\ref{wt}) we have
$$
x^{(n_0+1)}=1={ax^{(n_0)}y^{(n_0)}+ b(x^{(n_0)}(1-y^{(n_0)})+y^{(n_0)}(1-x^{(n_0)}))\over ax^{(n_0)}y^{(n_0)}+b(x^{(n_0)}(1-y^{(n_0)})+y^{(n_0)}(1-x^{(n_0)}))+c(1-x^{(n_0)})(1-y^{(n_0)})}
$$
or
$$y^{(n_0+1)}=1={\alpha x^{(n_0)}y^{(n_0)}+ \beta (x^{(n_0)}(1-y^{(n_0)})+
y^{(n_0)}(1-x^{(n_0)}))\over \alpha x^{(n_0)}y^{(n_0)}+\beta (x^{(n_0)}(1-y^{(n_0)})+y^{(n_0)}(1-x^{(n_0)}))+\gamma (1-x^{(n_0)})(1-y^{(n_0)})}.
$$
Consequently
$$c(1-x^{(n_0)})(1-y^{(n_0)})=0  \ \ \mbox{or} \ \ \gamma (1-x^{(n_0)})(1-y^{(n_0)})=0$$
this is impossible, because of $c\gamma\ne 0$ and (\ref{n0}).
\end{proof}

Denote
$$s={x\over 1-x}, \ \ t={y\over 1-y},$$
\begin{equation}\label{par}
A=a/c, \ \ B=b/c, \ \ C=\alpha/\gamma, \ \ D=\beta/\gamma.
\end{equation}
Then from (\ref{w}) we get
\begin{equation}\label{ww}
T:\left\{\begin{array}{ll}
s'={x'\over 1-x'}=A st+B (s+t),\\[3mm]
t'={y'\over 1-y'}=C st+D (s+t).
\end{array}
\right.
\end{equation}
Note that $A, B, C, D\geq 0$ and $T: [0,+\infty)^2\to [0,+\infty)^2$.

By the equality (\ref{wt}) and Lemma \ref{q2} we get the following relation between the
trajectory of the operator (\ref{w}) and the operator (\ref{ww}):

\begin{equation}\label{wwt}
T^{(n+1)}:\left\{\begin{array}{ll}
s^{(n+1)}={x^{(n+1)}\over 1-x^{(n+1)}}=A s^{(n)}t^{(n)}+B (s^{(n)}+t^{(n)}),\\[3mm]
t^{(n+1)}={y^{(n+1)}\over 1-y^{(n+1)}}=C s^{(n)}t^{(n)}+D (s^{(n)}+t^{(n)}).
\end{array}
\right.
\end{equation}
\begin{lemma}\label{ty} A trajectory of the operator (\ref{w}) converges if and only if the
corresponding trajectory (\ref{wwt}) of (\ref{ww}) converges.
\end{lemma}
\begin{proof} Follows from the relations
$$s^{(n)}={x^{(n)}\over 1-x^{(n)}}, \ \ t^{(n)}={y^{(n)}\over 1-y^{(n)}},$$
and their inverse:
\begin{equation}\label{tst}
x^{(n)}={s^{(n)}\over 1+s^{(n)}}, \ \ y^{(n)}={t^{(n)}\over 1+t^{(n)}}.
\end{equation}
\end{proof}
Fixed points of $T$ are
$$O=(0,0), \ \ P=\left({B+D-1\over (B-1)C-AD}, \, {B+D-1\over (D-1)A-BC}\right).$$
Denote
$$\mathcal P=\left\{(A,B,C,D)\in R^4_+: P>0, {\rm i.e.}, {B+D-1\over (B-1)C-AD}>0, \, {B+D-1\over (D-1)A-BC}>0\right\}.$$
It is easy to see that $\mathcal P$ is not empty set. Indeed, if $B+D<1$ then $P>0$:
$${B+D-1\over C(B-1)-AD}>0, \, {B+D-1\over A(D-1)-BC}>0.$$

Thus we have
\begin{lemma} The set of fixed points of the operator (\ref{ww}) is
$${\rm Fix}(T)=\left\{\begin{array}{ll}
\{O\}, \ \ \mbox{if} \ \  (A,B,C,D)\notin \mathcal P\\[2mm]
\{O, P\}, \ \ \mbox{if} \ \  (A,B,C,D)\in \mathcal P
\end{array}
\right.$$
\end{lemma}
To define the type of the fixed points consider the Jacobian of the operator $T$:
\begin{equation}\label{j}
J(s,t)=\left(\begin{array}{cc}
At+B & As+B\\[3mm]
Ct+D & Cs+D\\
\end{array}
\right)
\end{equation}

For the fixed point $O$ we have that $J(O)$ has two eigenvalues
$0$ and $B+D$. Therefore this point is attractor iff $B+D<1$;
non-hyperbolic iff $B+D=1$ and saddle iff $B+D>1$.

For the fixed point $P$ one can explicitly calculate eigenvalues $\lambda_1$, $\lambda_2$ of $J(P)$, but they have very long formulas.
Therefore we give these eigenvalues for concretely chosen parameters:\\
\begin{center}
\begin{tabular}{|l|l|l|l|l|l|l|}
\hline
A& B& C& D&$\lambda_1$&$\lambda_2$&type\\[2mm]
\hline
0.3&0.1&0.2&0.1&1.795344459& -0.5853490213&saddle\\[2mm]
 \hline
0.9&0.5&0.2&0.8&0.8055778837& -0.5430343293&attractor\\[2mm]
  \hline
9&5&2&0.8&12.72976779& 1.409215262&repeller\\[2mm]
\hline
9&0.5&2&0.5&0& 1&non-hyperbolic\\[2mm]
 \hline
\end{tabular}
\end{center}

By this table we see that the fixed point $P$ may have any possible types.

 From the known theorem about stable and unstable manifolds (see \cite{De} and  \cite{G})
 we get the following result
 \begin{pro}\label{pt} If parameters of the operator (\ref{ww}) are such that $O$ (resp. $P$) is
 \begin{itemize}
 \item[-] {\rm attractor} then there exists a neighborhood $U\subset [0,+\infty)^2$ of $O$ (resp. of $P$) such
 that
 $\lim_{n\to \infty}T^n(x)=O$ \ \ (resp. $=P$) \ \ \mbox{for all} \ \ $x\in U.$
 \item[-] {\rm saddle} then there exists an invariant\footnote{a curve $\gamma$ is invariant with respect to $T$ if $T(\gamma)\subset \gamma$.} curve $\xi$ (resp. $\eta$)
 through $O$ (resp. $P$) in the set $[0,+\infty)^2$ such that for any initial
 point $x\in \xi$ (resp. $\in \eta$) one has
 $\lim_{n\to\infty} T^n(x)=O$ \ \ (resp. $=P$).
 \item[-] if $P$ is {\rm repeller} then there exists an invariant curve $\zeta$
 through $P$ in the set $[0,+\infty)^2$ and there is  a neighborhood $\mathcal N(P)$ of $P$ such that
 for any initial point $v\in \zeta\cap \mathcal N(P)$, there exists $k=k(v)\in \mathbb N$ that $T^k(v)\notin \mathcal N(P)$.
 \end{itemize}
   \end{pro}
The curves $\xi, \eta$ are known as stable manifolds and $\zeta$ is an unstable manifold
(see \cite{Ga} for notations of stable manifold, stable eigenspace etc.)

The following theorem is corollary of the above proved results

\begin{thm}\label{tt} If parameters of the operator (\ref{w}) are such that $z_0$ (resp. $z_2$) is
 \begin{itemize}
 \item[-] {\rm attractor} then there exists a neighborhood $\mathcal U\subset S$ of $z_0$ (resp. of $z_2$) such
 that
 $\lim_{n\to \infty}W^n(x)=z_0$ \ \ (resp. $=z_2$) \ \ \mbox{for all} \ \ $x\in \mathcal U.$
 \item[-] {\rm saddle} then there exists an invariant curve $\tilde\xi$ (resp. $\tilde\eta$)
 through $z_0$ (resp. $z_2$) in the set $S$ such that for any initial
 point $x\in \tilde\xi$ (resp. $\in \tilde\eta$) one has
 $\lim_{n\to\infty} W^n(x)=z_0$ \ \ (resp. $=z_2$).
 \item[-] if $z_2$ is {\rm repeller} then there exists an invariant curve $\tilde\zeta$
 through $z_2$ in the set $S$ and there is  a neighborhood $\mathbf N(z_2)$ of $z_2$ such that
 for any initial point $v\in \zeta\cap \mathbf N(z_2)$, there exists $k=k(v)\in \mathbb N$ that $W^k(v)\notin \mathbf N(z_2)$.
 \end{itemize}
   \end{thm}
 \begin{proof} Formula (\ref{par}) gives relation between parameters of operator (\ref{ww}) and
(\ref{w}). Then formula (\ref{5}) shows that $O$ corresponds to $z_0=(0,0)$ and $P$ corresponds to $z_2$.
Consequently, by Lemma \ref{ty} it follows that $z_0$  is attractor iff $B+D={b\over c}+{\beta\over \gamma}<1$;
non-hyperbolic iff $B+D=1$ and saddle iff $B+D>1$. The fixed point $z_2$ (as $P$) may have any type
(attractor, saddle, repeller, non-hyperbolic).
Therefore by Proposition \ref{pt} and formulas (\ref{tst}) one completes the proof.
 \end{proof}

\section*{ Acknowledgements}

UAR thanks the universit\'e Paris Est Cr\'eteil and a program of LabEx Bezout (ANR-10-LABX-58) for supporting his visit to the University.

\end{document}